\documentclass[reqno]{amsart}

\usepackage{amssymb}
\usepackage{amsthm}
\usepackage{amsmath}
\usepackage{float}
\usepackage{fancyhdr}

\usepackage{tikz-cd}

\usepackage[matrix, arrow]{xy}
\xyoption{arrow}
\xyoption{curve}
\theoremstyle{plain}\newtheorem{Theorem}{Theorem}[section]
\theoremstyle{plain}
\theoremstyle{plain}\newtheorem{Corollary}[Theorem]{Corollary}
\theoremstyle{plain}\newtheorem{Lemma}[Theorem]{Lemma}
\theoremstyle{plain}\newtheorem{Proposition}[Theorem]{Proposition}
\theoremstyle{plain}
\theoremstyle{definition}
\theoremstyle{definition}
\theoremstyle{definition}
\theoremstyle{definition}
\theoremstyle{definition}
\theoremstyle{definition}\newtheorem{Remark}[Theorem]{Remark}
\theoremstyle{definition}
\theoremstyle{definition}

\def\CA{{\mathcal{A}}}  
\def\CB{{\mathcal{B}}}

\def\CK{{\mathcal{K}}}
\def\CL{{\mathcal{L}}}

\def\CP{{\mathcal{P}}}    

\def\CR{{\mathcal{R}}}

\def\CU{{\mathcal{U}}}

\def\CA{{\mathcal{A}}}
\def\CA{{\mathcal{A}}}

\def\F{{\mathbb F}}   
\def\Fp{{\mathbb F_p}}

\def\Aut{\mathrm{Aut}}               \def\tenk{\otimes_k}     
\def\Br{\mathrm{Br}}                 \def\ten{\otimes}

\def\chr{\mathrm{char}}

\def\dim{\mathrm{dim}}

\def\Der{\mathrm{Der}}

\def\Ext{\mathrm{Ext}}

\def\Hom{\mathrm{Hom}}

\def\IBr{\mathrm{IBr}}

\def\IDer{\mathrm{IDer}}

\def\mod{\mathrm{mod}}

\def\op{\mathrm{op}}

\def\Tr{\mathrm{Tr}}

\def\lmd{\lambda}
\def\tn{\textnormal}
\usepackage[hidelinks]{hyperref}
\usepackage{tikz-cd}

\def\hx{\hat{x}}

\def\Va{V_{x,\alpha}}

\def\lmd{\lambda}
\def\tn{\textnormal}
\usepackage[hidelinks]{hyperref}
\usepackage{tikz-cd}

\def\hx{\hat{x}}

\def\Va{V_{x,\alpha}}

\newtheorem*{theorem*}{Theorem}

\usepackage{caption}
\makeindex
\author{William Murphy} 
\date{\today}

\title[The $HH^1$ of the blocks of the Mathieu groups]{The Lie algebra structure of the $HH^1$ of the blocks of the sporadic Mathieu groups}

\begin{document}

\begin{abstract}
Let $G$ be a sporadic Mathieu group and $k$ an algebraically closed field of prime characteristic $p$, dividing the order of $G$. In this paper we describe some of the Lie algebra structure of the first Hochschild cohomology groups of the $p$-blocks of $kG$. In particular, letting $B$ denote a $p$-block of $kG$, we calculate the dimension of $HH^1(B)$ and in the majority of cases we determine whether $HH^1(B)$ is a solvable Lie algebra.
\end{abstract}

\vspace*{-8mm}
\maketitle

\vspace*{-4mm}
\section{Introduction}

Let $G$ be a finite group and $k$ an algebraically closed field of prime characteristic $p$ dividing the order of $G$. It is well known that the first Hochschild cohomology group (the $HH^1$) of the group algebra $kG$ form a Lie algebra over $k$. By a result of Fleischmann, Janiszczak and Lempken \cite{FJL}, this Lie algebra has non-zero dimension over $k$. On the other hand, the problem of determining whether $HH^1(B)\neq\{0\}$ for a $p$-block $B$ of $kG$ is still open in general. 

In this paper we settle this problem for the blocks of the sporadic Mathieu groups, in particular showing the following.

\begin{Theorem}\label{ThrmMain}
Let $G$ be a sporadic Mathieu group. Then for all primes $p$ dividing $|G|$ and for all $p$-blocks $B$ of $kG$ with a non-trivial defect group, $HH^1(B)\neq\{0\}$. Moreover, for each $B$ we determine the dimensions $\dim_k(HH^1(B))$.
\end{Theorem} 

\begin{Remark}
In the process of proving Theorem \ref{ThrmMain} we collect interesting information on the $p$-blocks of the sproadic Mathieu groups. Out of the collection of $40$ such $p$-blocks with non-zero defect, we determine that $37$ have $HH^1$ that is a solvable Lie algebra, and $1$ has $HH^1$ that is a simple Lie algebra. Of the $18$ $p$-blocks whose $HH^1$ is a $2$-dimensional Lie algebra, we determine that $15$ are non-abelian and $1$ is abelian. Moreover, we have that for $G=M_{11}$ or $M_{22}$, $HH^1(kG)$ is a solvable Lie algebra for all primes $p=\chr(k)$ dividing $|G|$.
\end{Remark}

\begin{Remark}
By Farrell and Kessar \cite[8.1(ii)]{FarKess} the Morita-Frobenius number for the blocks of the sporadic Mathieu groups is 1. Thus by \cite{KessarDonovan} any basic algebra $A$ of these blocks is defined over the prime field $\Fp$; $A$ has a $k$-basis whose structure constants are contained in $\Fp$, or equivalently for each block $B$ and basic algebra $A$ of $B$, there is an $\Fp$-algebra $A_0$ such that $A\cong k\ten_{\Fp} A_0$. Consequently, $HH^1(B)\cong HH^1(A)\cong k\ten_{\Fp}HH^1(A_0)$ is an isomorphism of Lie algebras; that is the Lie algebra $HH^1(B)$ is defined over $\Fp$.

Moreover, for $G=M_{11}$ or $M_{22}$ one can read of from Table \ref{tabdefect} that the blocks $B$ of $kG$ have a Frobenius number of 1, so that $B \cong k\tenk B_0$ for some $\Fp$-algebra $B_0$: $B$ itself is also defined over $\Fp$ and we see that in this case there is a Lie algebra isomorphism $HH^1(B)\cong k\ten_{\Fp}HH^1(B_0)$ .
\end{Remark}

In order to determine the dimensions of the $HH^1$ of the blocks of $kG$, the dimension of $HH^1(kG)$ is calculated, and given in Table \ref{tabdimkg}. In Section \ref{Mathieu} we prove our result using the centraliser decomposition (\cite[Theorem 2.11.2]{BensonII}, reviewed in Proposition \ref{CentraliserDecompProp}) and the $p$-block structure of each sporadic Mathieu group. In Section \ref{PrE} we determine in general the Lie algebra structure of $HH^1(B)$ for $B$ a block with a non-trivial cyclic defect group, giving an explicit formula for its dimension (though this is already folklore) and a characterisation of a basis for $HH^1(B)$ as a Lie algebra. This is a solvable Lie algebra in general (see \cite[Example 5.7]{LinRubII}), and we note that for some blocks $B$ with a non-trivial cyclic defect group, we see that $HH^1(B)$ is $1$-dimensional. In this case the we are also able to determine that there is  $p$-toral basis of $HH^1(B)$.

In Section \ref{altsum} we use an alternative method to determine the dimension of the $HH^1$ of the non-principal $2$-block of the sporadic Mathieu group $M_{12}$ and we hope to use this method in other cases going forwards. When $G=M_{12}$ and $p=2$, the non-principal block of $kG$ has a non-cyclic and non-trivial defect group; this is the only such case over all primes dividing the orders of all sporadic Mathieu groups $G$, and so the only case where the centraliser decomposition alone is insufficient to compute the dimensions of the $HH^1$. In particular we use a result of K\"{u}lshammer and Robinson \cite{KueRoHH} to verify this dimension. This result gives an alternating sum formula for the dimension of the Tate-Hochschild cohomology, equivalent to the Hochschild cohomology in positive degree, and to our knowledge this paper marks the first explicit use of this formula. Whilst a shorter proof exists to determine these dimensions (see Theorem \ref{ThrmM12}(i)), our verification showcases the power of this formula and its potential application to other finite groups that do not have such ``nice" $p$-block structure as the Mathieu groups (see the remarks following Table \ref{tabdefect}).

Appendix \ref{Appcode} gives the \texttt{GAP} code constructed to compute the dimensions of the $HH^1$ of (hypothetically) any finite group on which \texttt{GAP} can perform calculations, using the well-known centraliser decomposition.

We note that although the Hochschild cohomology of a $k$-algebra $A$ is defined by $HH^*(A)=\Ext^*_{A\ten_k A^{\op}}(A;A)$, the first Hochschild cohomology admits an alternative characterisation as the set of $k$-linear derivations on $A$, modulo inner derivations: $HH^1(A)\cong \Der(A)/\IDer(A)$. We will use this point of view throughout.

We will write by $C_n$, $S_n$, $A_n$ and $D_{2n}$ the cyclic, symmetric, alternating and dihedral groups of orders $n$, $n!$, $n!/2$ and $2n$ respectively, with $Q_{2^n}$ and $SD_{2^n}$ denoting the quaternion and semi-dihedral groups of order $2^n$. The projective special linear group of degree $n$ over the finite field with $q$ elements is denoted $L_n(q)$, and for $G$ a group acting on a set $X$, we will denote by $X^G$ the fixed points of $X$ under this action.

\section{The sporadic Mathieu groups}\label{Mathieu}

Let $G$ be a sporadic Mathieu group, and let $k$ be an algebraically closed field of prime characteristic $p$ dividing $|G|$. Throughout we use \texttt{GAP} \cite{GAP}, the ATLAS \cite{ATLAS}, the online group database found at \cite{DokGrp}, the Modular Atlas project \cite{ModAtlas} and its print counterpart \cite{AtlasBC} as our sources of data on the groups $G$, their subgroups and their blocks.

To calculate the dimensions of $HH^1(kG)$ over $k$ we use the centraliser decomposition. From this we construct a simple code in \texttt{GAP} (see Appendix \ref{Appcode}) which provides the following data on the dimensions of $HH^1(kG)$ at the different primes $p$.

\begin{table}[H]
\caption{ The dimensions of $HH^1(kG)$ for $p=\chr(k)$ dividing $|G|$. }\label{tabdimkg}
\centering
\begin{tabular}{l||l||l|l|l|l|l|l}
	$G$   & \hfil Order 											& $p=2$	& $p=3$	& $p=5$	& $p=7$	& $p=11$ & $p=23$ \\
  	\hline
  	\hline
$M_{11}$   &  $2^4 \cdot 3^2 \cdot 5 \cdot 11$			 			& $6$	& $2$	& $1$	& $-$	& $2$	 & $-$ \\
$M_{12}$   &  $2^6 \cdot 3^3 \cdot 5 \cdot 11$						& $15$	& $4$	& $2$	& $-$	& $2$	 & $-$ \\
$M_{22}$   &  $2^7 \cdot 3^2 \cdot 5 \cdot 7 \cdot 11$				& $9$	& $3$	& $1$	& $2$	& $2$	 & $-$ \\
$M_{23}$   &  $2^7 \cdot 3^2 \cdot 5\cdot7 \cdot 11 \cdot 23$			& $9$	& $5$	& $3$	& $4$	& $2$	 & $2$ \\
$M_{24}$   &  $2^{10} \cdot 3^3 \cdot 5 \cdot 7 \cdot 11 \cdot 23$	& $22$	& $10$	& $4$	& $6$	& $1$	 & $2$ 
\end{tabular} 
\end{table}

We need to describe how the dimensions of $HH^1(kG)$ given in Table \ref{tabdimkg} decompose into the dimensions of the $HH^1$ of the $p$-blocks of $kG$, in view of the structure provided in Table \ref{tabdefect} below. This table - the data of which one may verify in \texttt{GAP} - describes the $p$-block structure of $kG$ for the prime divisors $p$ of $|G|$, in particular the number of blocks and their defect. The notation $[n_0,n_1,\ldots,n_m]$ denotes that $G$ has $p$-block structure $kG=B_0\oplus B_1\oplus \cdots\oplus B_m$ where $B_0$ is the principal block of $kG$, $B_i$ has defect $n_i$, and the blocks are ordered according to the decreasing size of their defect. The notation $1^n$ (resp. $0^n$) denotes $n$ successive $p$-blocks whose defect groups are non-trivial cyclic (resp. trivial).

\begin{table}[H]
\caption{ The defect of the $p$-blocks of $kG$ for $p=\chr(k)$ dividing $|G|$.  }\label{tabdefect}
\centering
\begin{tabular}{l||l|l|l|l|l|l}
	   										$G$	& $p=2$	& $p=3$	& $p=5$	& $p=7$	& $p=11$ & $p=23$ \\
  	\hline
  	\hline
$M_{11}$   	& $[4,0^2]$	& $[2,0]$	& $[1,0^5]$	& $-$	& $[1,0^3]$	 & $-$ \\
$M_{12}$   	& $[6,2]$	& $[3,1,0]$	& $[1^2,0^5]$	& $-$	& $[1,0^8]$	 & $-$ \\
$M_{22}$   	& $[7]$	& $[2,1,0^3]$	& $[1,0^7]$	& $[1,0^7]$	& $[1,0^5]$	 & $-$ \\
$M_{23}$   	& $[7,0^2]$	& $[2,1,0^5]$	& $[1^2,0^8]$	& $[1^2,0^7]$	& $[1,0^{10}]$	 & $[1,0^4]$ \\
$M_{24}$   	& $[10]$	& $[3,1^4,0]$	& $[1^3,0^{12}]$	& $[1^3,0^{11}]$	& $[1,0^{15}]$	 & $[1,0^{13}]$ 
\end{tabular} 
\end{table}

The structure of the Mathieu groups lends itself to such calculations in general: for each sporadic Mathieu group $G$ and each prime $p$ dividing $|G|$, all non-principal $p$-blocks have a cyclic defect group aside from the aforementioned exception of $G=M_{12}$ and $p=2$ (see the the last remark of $\S1$ in \cite{AnCon}). The blocks of the Mathieu groups with a non-trivial cyclic defect group have structure that is well understood and the dimensions of the $HH^1$ is easily calculated in this case (see Section \ref{PrE}). On the other hand for the blocks with a trivial defect group the $HH^1$ vanishes: any such block is isomorphic to a matrix algebra over $k$ (see \cite{BraArith}), and the $HH^*$ of a separable algebra is zero \cite[Theorem 4.1]{Hoch}.

The dimensions of the $HH^1$ of the principal block $B_0$ are then readily calculated (aside from the noted exception) by comparing the dimensions of $HH^1(kG)$ and $HH^1(B)$ for the non-principal blocks $B$ with cyclic defect groups: $$\dim_k(HH^1(B_0))=\dim_k(HH^1(kG)) \hspace{3mm} -\sum_{\substack{B\neq B_0,\\ B \ \tn{has cyclic} \\ \tn{ defect group}}}\dim_k(HH^1(B)).$$

In the exceptional case, $G=M_{12}$ has a non-principal 2-block with Klein four defect group and so to calculate the dimensions and show solvability of the $HH^1$ of this block we use results of Holm \cite{Holm}, Erdmann \cite{Erdmanntame} and Linckelmann and Rubio y Degrassi \cite{LinRubII}. 

Since we are concerned with determining not just the dimensions of the $HH^1$ but its solvability as a Lie algebra, we will often make use of the following result.

\begin{Theorem}[\cite{LinRubII,RubSchroSol}]\label{ExtThrm}

Let $A$ be a finite
dimensional $k$-algebra such that $\dim_k(\Ext^1_
A(S,S)) = 0$ and $\dim_k(\Ext^1_A(S,T)) \leq 1$, for all non-isomorphic simple $A$-modules $S$ and $T$. Then $HH^1(A)$ is a solvable Lie algebra.

\end{Theorem}

We organise Theorem \ref{ThrmMain} into the following 5 results, categorised with regards to which Mathieu group they relate to. In each theorem, $k$ is an algebraically closed field of characteristic $p$ dividing $|G|$, $G$ a sporadic Mathieu group. We will write $B_0$ throughout for the principal $p$-block of $kG$, and $B_i$ for the non-principal $p$-blocks of $kG$, $i=1,2,\ldots,$ if they exist, ordered in descending size of their defect. For $P$ a cyclic $p$-group, and $E$ a cyclic $p'$-group acting on $P$, we let $HH^1(kP)^E$ denote the fixed points under the the action of $E$ on $HH^1(kP)$, induced by the action of $E$ on $P$  (see Section \ref{PrE} for details). We adopt the convention that the ``zero Lie algebra" $\CL=\{0\}$ with Lie bracket $[0,0]=0$ is trivially solvable.

\begin{Theorem}\label{ThrmM11}
Let $G=M_{11}$. Then $HH^1(kG)$ is a solvable Lie algebra isomorphic to $HH^1(B_0)$, every non-principal block of $kG$ has defect 0, and the following isomorphisms of Lie algebras with given dimensions hold. \begin{itemize}
\item[(i)] Let $p=2$ and $B_0(kL_3(q))$ be the principal $2$-block of $L_3(q)$, $q=3 \ (\mod 4)$. Then there are isomorphisms of $6$-dimensional solvable Lie algebras, $$HH^1(kG)\cong HH^1(B_0)\cong HH^1(B_0(kL_3(q))).$$ 
\item[(ii)] Let $p=3$. Then there are isomorphisms of $2$-dimensional Lie algebras $$HH^1(kG)\cong HH^1(B_0)\cong HH^1(k(C_3^2\rtimes SD_{16})).$$
\item[(iii)] Let $p=5$. Then there are isomorphisms of $1$-dimensional Lie algebras $$HH^1(kG)\cong HH^1(B_0)\cong HH^1(kC_5)^{C_4}\cong k.$$
\item[(iv)] Let $p=11$. Then there are isomorphisms of $2$-dimensional non-abelian Lie algebras $$HH^1(kG)\cong HH^1(B_0)\cong HH^1(kC_{11})^{C_5}.$$
\end{itemize}
\end{Theorem}

\begin{table}[H]
\centering
\caption{ Data on the $p$-blocks of $kM_{11}$.  }\label{tabM11}
\begin{tabular}{l|l|l|l|l}
$G=M_{11}$		& $p=2$	& $p=3$	& $p=5$	& $p=11$ \\
  	\hline
  	\hline
No. of blocks, $B_i$, of $kG$ & 3 & 2 & 6 & 4 \\
Defects of the $B_i$ 	& $[4,0^2]$	& $[2,0]$	& $[1,0^5]$	& $[1,0^3]$	 \\
Dimensions $\dim_k(HH^1(B_i))$ 	& $[6,0^2]$	& $[2,0]$	& $[1,0^5]$	& $[2,0^3]$ \\
Sylow $p$-subgroup $P$ of $G$  	& $SD_{16}$	& $C_3^2$	& $C_5$	& $C_{11}$ \\
$E=N_G(P)/PC_G(P)$  	& $\{1\}$	& $SD_{16}$	& $C_4$	& $C_5$	
\end{tabular} 
\end{table}

In Table \ref{tabM11} above, and Tables \ref{tabM12}, \ref{tabM22}, \ref{tabM23}, and \ref{tabM24} below, the notation of the second row (the defects of the blocks) is the same notation as in Table \ref{tabdefect}. The notation of the third row is similar: $[n_0,n_1,\ldots,n_m]$ denotes that $G$ has $p$-block structure $kG=B_0\oplus B_1\oplus\cdots\oplus B_m$ where $HH^1(B_i)$ has dimension $n_i$, whilst $n^m$ denotes $m$ successive $p$-blocks whose $HH^1$ has dimension $n$.

\begin{proof}

Let $G=M_{11}$. The proof will refer to facts included in Section \ref{PrE}.

By the data in Table \ref{tabM11} one sees that for all primes dividing $|G|$, the block structure of $kG$ provides the isomorphism $HH^1(kG)\cong HH^1(B_0)$. The dimensions $\dim_k(HH^1(B_0))$ are therefore simply the same as the dimensions of $HH^1(kG)$, which are calculated using the centraliser decomposition.

Let $p=2$. A Sylow 2-subgroup of $G$ is isomorphic to the semidihedral group of order 16, $SD_{2^n}=SD_{16}$, so that $B_0$ is a tame block. The decomposition matrix of $B_0$ is available on \texttt{GAP} allowing confirmation that $B_0$ has 3 irreducible Brauer characters, and so 3 simple modules up to isomorphism. Using Erdmann's classification of tame blocks \cite[pp.300]{Erdmanntame} one verifies that the decomposition and Cartan matrices of $B_0$ match those of $SD(3\CB)_1$. This may be further verified in \cite[3.3/3.4]{Holm} which gives us explicitly a derived equivalence between $B_0$ and $B_0(L_3(q)), q=3 \ (\mod 4)$. The quiver and relations of $B_0$ given by Erdmann enable us to determine that $B_0$ is non-local, which thanks to a result of Eisele and Raedschelders \cite[4.7]{EisRad}
provides the solvability of $HH^1(B_0)$. This proves (i).

Now let $p=3$. A Sylow 3-subgroup $P$ is isomorphic to $C_3\times C_3$, with inertial quotient $E=SD_{16}$ as in Table \ref{tabM11}. By a result of Rouquier \cite{Rouqstable} there is a stable equivalence of Morita type between $B_0$ and its Brauer correspondent $kN_G(P)=P\rtimes E$, (in fact there is derived equivalence between $B_0$ and $kN_G(P)$ constructed explicitly by Okuyama in an unpublished result \cite{OkDerEq}). Whence $HH^1(B_0)\cong HH^1(k(P\rtimes E))$ proving (ii). For (iii) and (iv), Theorems \ref{stabeqThrm}(iii) and \ref{fixptisom} provide the second isomorphism, and using Corollary \ref{dimprop} one may verify the dimensions of both match those found by the centraliser decomposition, completing the proof.
\end{proof}

\begin{Remark}
In the sequel (Theorem \ref{preformula}) we will see that one can write down explicitly a $k$-basis and Lie bracket for the Lie algebra $HH^1(kP)^{E}$ encountered in Theorem \ref{ThrmM11}(iii) and (iv), and that in fact this is possible for all $p$-blocks of finite groups with a non-trivial cyclic defect group $P\cong C_p$.
\end{Remark}

\begin{Theorem}\label{ThrmM12}
Let $G=M_{12}$. Then the following isomorphisms of Lie algebras with given dimensions hold.
\begin{enumerate}
\item[(i)] Let $p=2$. Then $HH^1(B_0)$ is a $13$-dimensional
Lie algebra and there is an isomorphism of $2$-dimensional abelian Lie algebras $$HH^1(B_1)\cong HH^1(kA_4).$$
\item [(ii)] Let $p=3$. Then $HH^1(B_0)$ is a $3$-dimensional solvable Lie algebra and there are isomorphisms of $1$-dimensional Lie algebras $$HH^1(B_1)\cong HH^1(kS_3)\cong k.$$ 

\item [(iii)] Let $p=5$. Then there are isomorphisms of $1$-dimensional Lie algebras $$HH^1(B_0)\cong HH^1(B_1)\cong HH^1(C_5)^{C_4}\cong k.$$ 
\item [(iv)] Let $p=11$. Then there are isomorphisms of $2$-dimensional non-abelian Lie algebras $$HH^1(kG)\cong HH^1(B_0)\cong HH^1(kC_{11})^{C_5}.$$
\end{enumerate}
\end{Theorem}

\begin{Remark}\label{UsRmrk}
Let $p=3$. By Usami \cite[Proposition 4]{Usami} there is a perfect isometry between the principal blocks of $kM_{12}$ and $L_3(3)$, and we note that the $HH^1$ of both is a 3-dimensional solvable Lie algebra. The dimension of the $HH^1$ in the $L_3(3)$ case is computed by the centraliser decomposition, and to see solvability in this case we apply a result of Koshitani \cite[Theorem 1]{Kosh} to Theorem \ref{ExtThrm}. On the other hand we are not aware of the existence of a derived equivalence in the literature and so we are unable to determine whether $HH^1(B_0)$ and $HH^1(kL_3(3))$ are isomorphic as Lie algebras.
\end{Remark}

\begin{table}[H]
	\centering
	\caption{ Data on the $p$-blocks of $kM_{12}$ }\label{tabM12}
	\begin{tabular}{l|l|l|l|l}
		$G=M_{12}$		& $p=2$	& $p=3$	& $p=5$	& $p=11$  \\
		\hline
		\hline
		No. of blocks, $B_i$ of $kG$ & 2 & 3 & 7 & 9 \\
		Defects of the $B_i$ 				&  $[6,2]$	& $[3,1,0]$	& $[1^2,0^5]$	& $[1,0^8]$	   \\
		Dimensions $\dim_k(HH^1(B_i))$ 		&  $[13,2]$	& $[3,1,0]$	& $[1^2,0^5]$	& $[2,0^8]$    \\
		Sylow $p$-subgroup $P$ of $G$  		&  $C_4^2\rtimes C_2^2$	& $3^{1+2}_+$	& $C_5$ 	& $C_{11}$  \\
		$E=N_G(P)/PC_G(P)$  	& $\{1\}$	& $C_2\times C_2$		& $C_4$			& $C_5$				\end{tabular} 
\end{table}

\begin{proof}
As before, the proof refers to results from Section \ref{PrE}. 

Let $p=2$. A defect group of $B_1$ is isomorphic to the Klein 4 group (see \cite{AnCon}), so that $B_1$ is a tame block. The decomposition matrix of $B_1$ is available on \texttt{GAP} allowing confirmation that $B_1$ has 3 irreducible Brauer characters, and so 3 simple modules up to isomorphism. Using Erdmann's classification of tame blocks \cite[pp.296]{Erdmanntame} one verifies that the decomposition and Cartan matrices of $B_1$ match those of $D(3\CK)$. We note that by \cite[Proposition 5.3]{LinRubII} this implies that $HH^1(B_1)$ is a solvable Lie algebra. The isomorphism and dimension of $HH^1(B_1)$ is then given by Holm \cite[2.1(2)]{Holm}, and by Table \ref{tabdimkg} we recover $\dim_k(HH^1(B_0))=13.$ 

To show that $HH^1(B_1)$ is the abelian $2$-dimensional Lie algebra, we use a technique that will occur again in the sequel (c.f. the proofs of Theorems \ref{ThrmM22}(ii) and \ref{preformula}) to find a basis for $HH^1(k(C_2\times C_2))^{C_3}$, which we will see is isomorphic to $HH^1(kA_4)$. Let $P=C_2\times C_2$ have generators $x$ and $y$, and $E=C_3$ have a generator $r$ so that $A_4\cong P\rtimes E$ with $r\cdot x=xy$ and $r\cdot y =x$. One verifies that the action of $E$ on $P$ is Frobenius, so the inclusion of $HH^1(kP)^E$ in $HH^1(k(P\rtimes E))$ (see Theorem \ref{fixptisom} below) is an isomorphism. We then use a result of Le and Zhou \cite[Corollory 3.4]{LeZhou} to find a basis of size $8$ for $HH^1(kP)$, and a result of Linckelmann \cite[Proposition 2.5.10]{LinBlockI} to find a basis of size $2$ for $HH^1(kP)^{E}$ and in particular we show that this basis has a trivial Lie bracket.

Let $P_1=\langle x\rangle$ and $P_2=\langle y\rangle$. Then the result \cite[Corollory 3.4]{LeZhou} tells us that there is an isomorphism of Lie algebras $HH^1(kP)\cong HH^1(kP_1)\tenk kP_2\oplus kP_1\tenk HH^1(kP_2).$ Let $d_i$ (resp. $d_i'$) be the derivation on $kP_1$ (resp. $kP_2$) sending $x\mapsto x^i$ (resp. $y\mapsto y^i$) so that $\{d_0,d_1\}$ is a $k$-basis for $HH^1(kP_1)$ as a Lie algebra (resp. $\{d_0',d_1'\}$ for $HH^1(kP_2)$). A basis for $HH^1(kP)$ is then given by $X=\{d_i\ten y^j, x^j\ten d_i' \ | \ i,j=0,1\}$, where $d_i\ten y^j$ sends $x\mapsto x^iy^j$ and $y\mapsto 0$, whilst $x^j\ten d_i'$ sends $x\mapsto 0$ and $y\mapsto x^jy^i$. Now let $f_x=d_0\ten y$ and $f_y=x\ten d_0'$. Using the result \cite[Proposition 2.5.10]{LinBlockI} as inspiration, we then look for a basis of $HH^1(kP)^E$ by considering the $E$-orbit sums of $X$ under the induced action on $HH^1(kP)$, $\ ^zD(g)=z\cdot D(z^{-1}\cdot g)$ for all $z\in E, g\in P$ and $D\in HH^1(kP)$.
	
We denote by $\Tr_1^E:HH^1(kP)\to HH^1(kP)^E$ the relative trace map, and define $D_z=\Tr_1^E(f_z)$ for $z=x$ or $y$. This gives the following $4$ maps; $D_x$ sends $x\mapsto 1+y$ and $y\mapsto 1+xy$, whilst $D_y$ sends $x\mapsto 1+xy$ and $y\mapsto 1+x,$ and one checks that $\{D_x, D_y\}$ forms a Lie algebra basis for $HH^1(kP)^E$. A simple verification shows that $[D_x,D_y]=0$, completing the proof of (i).

For the blocks in (iii) and (iv), and the non-principal block $B_1$ in (ii), Theorems \ref{stabeqThrm}(iii) and \ref{fixptisom} provide the isomorphisms, and Corollary \ref{dimprop} provides the dimensions. The dimension of the $HH^1$ of the principal block in (ii) is then verified using Table \ref{tabdimkg}. To see that this principal block has $HH^1$ a solvable Lie algebra, we apply results of Koshitani and Waki \cite[Theorem 1]{KoshWak} to Theorem \ref{ExtThrm}: the former gives us the dimensions $\dim_k(\Ext_{kG}^1(S,T))$ for all simple $kG$-modules $S$ and $T$, and this meets the solvability criteria of the latter.

\end{proof}

\begin{Theorem}\label{ThrmM22}
Let $G=M_{22}$. Then $HH^1(kG)$ is a solvable Lie algebra, and the following isomorphisms of Lie algebras with given dimensions hold.
\begin{enumerate}
\item[(i)] Let $p=2$. Then $kG$ is indecomposable and $HH^1(kG)$ is a $9$-dimensional solvable Lie algebra.
\item [(ii)] Let $p=3$. Then there is an isomorphism of $2$-dimensional non-abelian Lie algebras $$HH^1(B_0)\cong HH^1(k(C_3^2\rtimes Q_8)),$$ and there are isomorphisms of $1$-dimensional Lie algebras $$HH^1(B_1)\cong HH^1(kS_3)\cong k.$$
\item [(iii)] Let $p=5$. Then there are isomorphisms of $1$-dimensional Lie algebras $$HH^1(kG)\cong HH^1(B_0)\cong HH^1(C_5)^{C_4}\cong k.$$
\item [(iv)] Let $p=7$. Then there are isomorphisms of $2$-dimensional non-abelian Lie algebras $$HH^1(kG)\cong HH^1(B_0)\cong HH^1(kC_7)^{C_3}.$$
\item [(v)] Let $p=11$. Then there are isomorphisms of $2$-dimensional non-abelian Lie algebras $$HH^1(kG)\cong HH^1(B_0)\cong HH^1(kC_{11})^{C_5}.$$ 
\end{enumerate}
\end{Theorem}

\begin{table}[H]
	\centering
		\caption{ Data on the $p$-blocks of $kM_{22}$.  }\label{tabM22}
	\begin{tabular}{l|l|l|l|l|l}
		$G=M_{22}$		& $p=2$	& $p=3$	& $p=5$	& $p=7$ & $p=11$ \\
		\hline
		\hline
		No. of blocks, $B_i$ of $kG$ & 1 & 5 & 8 & 8 & 6 \\
		Defects of the $B_i$	& $[7]$	& $[2,1,0^3]$	& $[1,0^7]$	& $[1,0^7]$ & $[1,0^5]$	 \\
		Dimensions $\dim_k(HH^1(B_i))$ 	& $[9]$	& $[2,1,0^3]$	& $[1,0^7]$	& $[2,0^7]$ & $[2,0^5]$ \\
		Sylow $p$-subgroup $P$ of $G$  	& $C_2^4\rtimes D_8$	& $C_3^2$	& $C_5$ & $C_7$	& $C_{11}$ \\
		$E=N_G(P)/PC_G(P)$  	& $\{1\}$	& $Q_8$	& $C_4$	& $C_3$ & $C_5$	
	\end{tabular} 
\end{table}

\begin{proof}
	Let $G=M_{22}$. As before the proof refers to Section \ref{PrE}.
	
	By the data in Table \ref{tabM22} one sees that for all primes $p\neq3$ dividing $|G|$ the block structure of $kG$ provides the isomorphism $HH^1(kG)\cong HH^1(B_0)$. The dimensions $\dim_k(HH^1(B_0))$ are therefore simply the same as the dimensions of $HH^1(kG)$, which are calculated using the centraliser decomposition. 

    In the case $p=2$, to see that $HH^1(kG)$ is a solvable Lie algebra, we use the quiver for $B_0$ constructed by Hoffman \cite{Hoff} over $\F_2$ along with the fact that the simple $kG$-modules in $B_0$ have as a splitting field $\F_2$ (see \cite{James}) and the result follows by Theorem \ref{ExtThrm}. We note that in \cite{LemSta} Lempken and Staszewski construct the socle series for the projective indecomposable modules in the principal block of the 3-fold cover of $M_{22}$, from which the result may also be deduced.
	
	In the case $p=3$ we need only use Theorems \ref{stabeqThrm} and \ref{fixptisom}, and Corollary \ref{dimprop} to determine the isomorphisms $HH^1(B_1)\cong HH^1(kS_3)\cong HH^1(kC_3)^{C_2}$ and the Lie algebra dimension, which via Table \ref{tabdimkg} allows us to calculate $\dim_k(HH^1(B_0))=\dim_k(HH^1(kG))-\dim_k(HH^1(B_1))=2$.  A Sylow 3-subgroup $P$ is isomorphic to $C_3\times C_3$, with inertial quotient $E=Q_8$. As before, a result of Rouquier then \cite{Rouqstable} gives us a stable equivalence of Morita type between $B_0$ and its Brauer correspondent $kN_G(P)=P\rtimes E$, (and again there is derived equivalence between $B_0$ and $kN_G(P)$ constructed explicitly by Okuyama in his unpublished result \cite{OkDerEq}). 
	
	To show that $HH^1(B_0)$ is the unique (up to isomorphism) non-abelian 2-dimensional Lie algebra, we generalise some ideas that will be seen in the sequel, regarding blocks with a cyclic defect group, in particular Theorem \ref{preformula}. One verifies that the action of $E$ on $P$ is Frobenius, so the inclusion of $HH^1(kP)^E$ in $HH^1(k(P\rtimes E))$ (Theorem \ref{fixptisom} below) is an equality. We use a result of Le and Zhou \cite[Corollory 3.4]{LeZhou} to find a  basis of size $18$ for $HH^1(kP)$, and a result of Linckelmann \cite[Proposition 2.5.10]{LinBlockI} to find a basis of size $2$ for $HH^1(kP)^{E}$, and in particular we show that this basis has a nontrivial Lie bracket.
	
Let $P$ have as generators $x$ and $y$, and $E$ be generated by $r$ and $s$ such that $sr=r^3s$, and write the action of $E$ on $P$ as $r\cdot x=y^2$, $r\cdot y=x$, $s\cdot x=x^2y$ and $s\cdot y=xy$. Let $P_1=\langle x\rangle$ and $P_2=\langle y\rangle$. Then the result \cite[Corollory 3.4]{LeZhou} tells us that there is an isomorphism of Lie algebras $HH^1(kP)\cong HH^1(kP_1)\tenk kP_2\oplus kP_1\tenk HH^1(kP_2).$ Let $d_i$ (resp. $d_i'$) be the derivation on $kP_1$ (resp. $kP_2$) sending $x\mapsto x^i$ (resp. $y\mapsto y^i$) so that $\{d_0,d_1,d_2\}$ is a $k$-basis for $HH^1(kP_1)$ as a Lie algebra (resp. $\{d_0',d_1',d_2'\}$ for $HH^1(kP_2)$). A basis for $HH^1(kP)$ is then given by $X=\{d_i\ten y^j, x^j\ten d_i' \ | \ i,j=0,1,2\}$, where $d_i\ten y^j$ sends $x\mapsto x^iy^j$ and $y\mapsto 0$, whilst $x^j\ten d_i'$ sends $x\mapsto 0$ and $y\mapsto x^jy^i$. Now let $f_x=d_0\ten y$ and $f_y=x\ten d_0'$. Using the result \cite[Proposition 2.5.10]{LinBlockI} as inspiration, we look for a basis of $HH^1(kP)^E$ by considering the $E$-orbit sums of $X$ under the induced action on $HH^1(kP)$, $\ ^zD(g)=z\cdot D(z^{-1}\cdot g)$ for all $z\in E, g\in P$ and $D\in HH^1(kP)$.
	
We denote by $\Tr_1^E:HH^1(kP)\to HH^1(kP)^E$ the relative trace map, and define $D_z=\Tr_1^E(f_z)$ for $z=x$ or $y$. This gives the following. The map $D_x$ sends $x\mapsto 1+y+2x^2+2xy+xy^2+2x^2y^2$ and $y\mapsto1+x^2+2y^2+xy+2x^2y+2xy^2$, whilst $D_y$ sends $x\mapsto1+2x^2+y^2+xy+2xy^2+2x^2y$ and $y\mapsto1+x+2y^2+2xy+x^2y+2x^2y^2$ and one checks that $\{D_x,D_y\}$ does indeed form a Lie algebra basis for $HH^1(kP)^E$. A simple verification shows that $[D_x,D_y]=D_y-D_x$, completing the $p=3$ case.

For $p\in\{5,7,11\}$, Theorems \ref{stabeqThrm}(iii) and \ref{fixptisom} provide the second isomorphism and Corollary \ref{dimprop} provides the dimensions.
\end{proof}

\begin{Theorem}\label{ThrmM23}
Let $G=M_{23}$. Then the following isomorphisms of Lie algebras with given dimensions hold.
\begin{enumerate}
\item[(i)] Let $p=2$. Then there is an isomorphism of solvable $9$-dimensional Lie algebras $$HH^1(kG)\cong HH^1(B_0).$$
\item [(ii)] Let $p=3$. Then there is an isomorphism of $2$-dimensional Lie algebras  $$HH^1(B_0)\cong HH^1(k(C_3^2\rtimes SD_{16})),$$ and an isomorphism of $3$-dimensional simple Lie algebras $$HH^1(B_1)\cong HH^1(kC_3).$$ 
\item [(iii)] Let $p=5$. Then there are isomorphisms of $1$-dimensional Lie algebras $$HH^1(B_0)\cong HH^1(kC_5)^{C_4},$$ and there is an isomorphism of $2$-dimensional non-abelian Lie algebras $$HH^1(B_1)\cong HH^1(kC_5)^{C_2}.$$ 
\item [(iv)] Let $p=7$. Then there are isomorphisms of $2$-dimensional non-abelian Lie algebras $$HH^1(B_0)\cong HH^1(B_1)\cong HH^1(kC_7)^{C_3}.$$ 
\item [(v)] Let $p=11$. Then there are isomorphisms of $2$-dimensional non-abelian Lie algebras $$HH^1(kG)\cong HH^1(B_0)\cong HH^1(kC_{11})^{C_5}.$$ 
\item [(vi)] Let $p=23$. Then there are isomorphisms of $2$-dimensional non-abelian Lie algebras $$HH^1(kG)\cong HH^1(B_0)\cong HH^1(kC_{23})^{C_{11}}.$$ 
\end{enumerate}
\end{Theorem}

\begin{table}[H]
	\centering
	\caption{ Data on the $p$-blocks of $kM_{23}$ }\label{tabM23}
	\begin{tabular}{l|l|l|l|l|l|l}
		$G=M_{23}$		& $p=2$	& $p=3$	& $p=5$	& $p=7$ & $p=11$ & $p=23$ \\
		\hline
		\hline
		No. of blocks, $B_i$ of $kG$ & 3 & 7 & 10 & 9 & 11 & 14 \\
		Defects of the $B_i$ 				&  $[7,0^2]$			& $[2,1,0^5]$	& $[1^2,0^8]$	& $[1^2,0^7]$	& $[1,0^{10}]$	 & $[1,0^4]$ \\
		Dimensions $\dim_k(HH^1(B_i))$ 		&  $[9,0^2]$			& $[2,3,0^5]$	& $[1,2,0^8]$	& $[2,2,0^7]$   & $[2,0^{10}]$   & $[2,0^4]$ \\
		Sylow $p$-subgroup $P$ of $G$  					& $C_2^4\rtimes D_8$	& $C_3^2$		& $C_5$ 		& $C_7$			& $C_{11}$		 & $C_{23}$  \\
		$E=N_G(P)/PC_G(P)$  	& $\{1\}$				& $SD_{16}$		& $C_4$			& $C_3$			& $C_5$			 & $C_{11}$
	\end{tabular} 
\end{table}

\begin{Remark}
Although the principal $2$-block of $kM_{22}$ and $kM_{23}$  have isomorphic defect groups and $HH^1$ of equal dimension, they are not derived equivalent: $kM_{22}$ and $kM_{23}$ have centres of dimension $12$ and $17$ respectively, whence the principal blocks have centres of dimension $12$ and $15$ respectively.
\end{Remark}

\begin{proof}
	Let $G=M_{23}$. The proof refers to Section \ref{PrE}.
	
	By the data in Table \ref{tabM23} one sees that for $p\in\{2,11,23\}$, the block structure of $kG$ provides the isomorphism $HH^1(kG)\cong HH^1(B_0)$ and so the dimensions $\dim_k(HH^1(B_0))$ are therefore the same as the dimensions of $HH^1(kG)$, calculated using the centraliser decomposition. For the primes $p=11$ and $p=23$ this can be verified using the isomorphisms of Theorems \ref{stabeqThrm}(iii) and \ref{fixptisom}, followed by Corollary \ref{dimprop}. These results also complete the case for $p=5$ or $7$.
	
	To see that the principal $2$-block has $HH^1$ a solvable Lie algebra we use \cite[Theorem 4.7]{LuxWieg} which constructs the projective indecomposable $\F_2G$-modules and shows us that $\Ext_{\F_2G}^1(S,S)=\{0\}$ and $\dim_{\F_2}(\Ext_{\F_2G}^1(S,T))\leq 1$ for all simple $\F_2G$-modules $S$ and $T$. The field $\F_2$ is a splitting field for the simple $kG$-modules in $B_0$ (see \cite{James}) and so by Theorem \ref{ExtThrm} we get that $HH^1(B_0)$ is solvable.

	Now let $p=3$. A Sylow 3-subgroup $P$ is isomorphic to $C_3\times C_3$, with inertial quotient $E=SD_{16}$ as in Table \ref{tabM23}. By a result of Rouquier \cite{Rouqstable} there is a stable equivalence of Morita type between $B_0$ and its Brauer correspondent $kN_G(P)=P\rtimes E$, (and as before there is derived equivalence between $B_0$ and $kN_G(P)$ constructed explicitly by Okuyama in an unpublished result). Whence $HH^1(B_0)\cong HH^1(k(P\rtimes E))$, the dimensions of which may be calculated using the centraliser decomposition. The non-principal block $B_1$ with non-trivial defect has $1$ Brauer character, so that the inertial quotient of $B_1$ is trivial, and by Theorem \ref{stabeqThrm} the given isomorphism holds.
	
\end{proof}

\begin{Remark}
Let $p=3$. The isomorphism of $3$-dimensional simple Lie algebras $HH^1(B_1)\cong HH^1(kC_3)$ agrees with results of Waki \cite[Theorem 2.4]{Waki} and Linckelmann and Rubio y Degrassi \cite[Theorem 1.3]{LinRubII}: the former shows that $B_1$ contains a unique isomorphism class of simple $kG$-modules with representative $S$, of dimension $231$, whose Loewy structure gives $\dim_k(\Ext_{kG}^1(S,S))=1$. The latter then tells us that if $HH^1(B_1)$ is simple, then it must be nilpotent. 
\end{Remark}

\begin{Theorem}\label{ThrmM24}
Let $G=M_{24}$.  Then the following isomorphisms of Lie algebras with given dimensions hold. 
\begin{enumerate}
\item[(i)] Let $p=2$. Then $kG$ is indecomposable and $HH^1(kG)$ is a $22$-dimensional solvable Lie algebra.
\item [(ii)] Let $p=3$. Then $\dim_k(HH^1(B_0))=6$ and there are isomorphisms of $1$-dimensional Lie algebras $$HH^1(B_i)\cong HH^1(kS_3)\cong k$$ for $i=1,\ldots,4$.

\item [(iii)] Let $p=5$. Then there are isomorphisms of $1$-dimensional Lie algebras $$HH^1(B_0)\cong HH^1(B_1)\cong HH^1(kC_5)^{C_4}\cong k,$$ and there is an isomorphism of $2$-dimensional non-abelian Lie algebras $$HH^1(B_2)\cong HH^1(kC_5)^{C_2}.$$ 
\item [(iv)] Let $p=7$. Then there are isomorphisms of $2$-dimensional non-abelian Lie algebras $$HH^1(B_0)\cong HH^1(B_1)\cong HH^1(B_2)\cong HH^1(kC_7)^{C_3}.$$
\item [(v)] Let $p=11$. Then there are isomorphisms of $1$-dimensional Lie algebras $$HH^1(kG)\cong HH^1(B_0)\cong k.$$
\item [(vi)] Let $p=23$. Then there are isomorphisms of $2$-dimensional non-abelian Lie algebras $$HH^1(kG)\cong HH^1(B_0)\cong HH^1(kC_{23})^{C_{11}}.$$
\end{enumerate}
\end{Theorem}

\begin{Remark}
Let $p=3$. Similar to Remark \ref{UsRmrk}, the result of Usami \cite[Proposition 4]{Usami} gives a perfect isometry between the principal blocks of $kM_{24}$ and $He$, the Held group. On the other hand no derived equivalence seems to exist in the literature and so we cannot determine any further the Lie algebra structure of $HH^1(B_0(kM_{24}))$
\end{Remark}

\begin{table}[H]
	\centering
		\caption{ Data on the $p$-blocks of $kM_{24}$   }\label{tabM24}
		\begin{tabular}{l|l|l|l|l|l|l}
		$G=M_{24}$		& $p=2$	& $p=3$	& $p=5$	& $p=7$ & $p=11$ & $p=23$ \\
		\hline
		\hline
		No. of blocks, $B_i$, of $kG$ & 1 & 8 & 15 & 14 & 16 & 14 \\
		Defects of the $B_i$ 	&  $[10]$				& $[3,1^4,0]$	& $[1^3,0^{12}]$	& $[1^3,0^{11}]$	     & $[1,0^{15}]$	 & $[1,0^{13}]$ \\
		Dimensions  $\dim_k(HH^1(B_i))$ 		&  $[22]$				& $[6,1^4,0]$	& $[1^2,2,0^{12}]$	& $[2^3,0^{11}]$  	 & $[1,0^{15}]$  & $[2,0^{13}]$ \\
		Sylow $p$-subgroup $P$ of $G$ 	&  \tn{Order 1024} & $3_+^{1+2}$ & $C_5$ & $C_7$ & $C_{11}$ & $C_{23}$  \\
		$E=N_G(P)/PC_G(P)$	& $\{1\}$				& $D_{8}$		& $C_4$			& $C_3$			& $C_{10}$			 & $C_{11}$
	\end{tabular} 
\end{table}

\begin{proof}
	Let $G=M_{24}$. For $p=2$ the case is clear, and we use the centraliser decomposition to compute the dimension of $HH^1(B_0)$. We note that from \texttt{GAP} a Sylow $2$-subgroup of $G$ has a structural description isomorphic to $(((D_8^2\rtimes C_2)\rtimes C_2)\rtimes C_2)\rtimes C_2$.
	
	To see that the principal $2$-block has $HH^1$ a solvable Lie algebra we use \cite[Section 5]{Selin} which provides the socle series' for the projective indecomposable $\F_2G$-modules, from which it may be read that $\Ext_{\F_2G}^1(S,S)=\{0\}$ and $\dim_{\F_2}(\Ext_{\F_2G}^1(S,T))\leq 1$ for all simple $\F_2G$-modules $S$ and $T$. By James \cite{James}, $\F_2$ is a splitting field for $G$ and so by Theorem \ref{ExtThrm}, $HH^1(B_0)$ is solvable.
	
	Now let $p=3$. A Sylow 3-subgroup $P$ is isomorphic to $3_+^{1+2}$, the extraspecial group of order 27 with exponent 3. The blocks $B_i$, $i=1,\ldots,4$ have 2 Brauer characters so that the inertial quotients are isomorphic to $C_2$ in this case, and the isomorphisms and given dimensions hold by Theorems \ref{stabeqThrm}(iii) and \ref{fixptisom}, and Corollary \ref{dimprop}. We calculate $\dim_k(HH^1(B_0))=\dim_k(HH^1(kG))-\sum_{i=1}^4\dim_k(HH^1(B_i))$ using the centraliser and block decomposition of $kG$. This proves (ii).
	
	For all remaining primes, we use the block structure and centraliser decomposition to determine a counting argument, the number of Brauer characters to determine the cyclic inertial quotient $E$, Theorems \ref{stabeqThrm}(iii) and \ref{fixptisom}, and Corollary \ref{dimprop} to get the Lie algebra isomorphisms and dimensions of the $HH^1$ of the blocks.
	
\end{proof}

\begin{proof}[Proof of Theorem \ref{ThrmMain}]
The result follows from Theorems \ref{ThrmM11}, \ref{ThrmM12}, \ref{ThrmM22}, \ref{ThrmM23} and \ref{ThrmM24}.
\end{proof}

\begin{Remark}
If $G$ is a sporadic Mathieu group of order divisible by $p$ then by a result of Fleischmann, Janiszczak and Lempken \cite{FJL} $G$ has a $p$-element $x$ such that $x\notin C_G(x)'$, the commutator subgroup of $C_G(x)$. As already stated this implies $HH^1(kG)\neq 0$ by the centraliser decomposition, for if $x\notin C_G(x)'$ then $C_G(x)/C_G(x)'$ has a direct factor that is a non-trivial $p$-group so that $\Hom(C_G(x),k)\neq0$ by Lemma \ref{quotLemma}. 

On the other hand, \cite{FJL} also implies that for all 2-cocycles $\alpha\in Z^2(G;k^\times)$ the Hochschild cohomology group $HH^1(k_\alpha G)\neq 0$; we have a twisted group algebra analogue to the centraliser decomposition (Proposition \ref{twistedcent}) and for any group $H$, if $g\in H$ is a $p$-element then $\alpha(h,g)=\alpha(g,h)$ for all $h\in C_H(g)$ (see \cite[2.6.1(iv)]{Kar}) so that the action of $C_H(g)$ on $V_{g,\alpha}$ in the twisted centraliser decomposition of $HH^1(k_\alpha H)$ is trivial. In particular for each sporadic Mathieu group $G$ there is a $p$-element $x\in G$ not contained in $C_G(x)'$ as determined in the result \cite{FJL}, and consequently there is a direct summand equal to $H^1(C_G(x);\Va)$ in the twisted centraliser decomposition of $HH^1(k_\alpha G)$ such that $H^1(C_G(x);\Va)\cong \Hom(C_G(x),k)\neq0$.

We note that more work is necessary to determine the block structure of the $k_\alpha G$ and the dimensions of the $HH^1$ of $k_\alpha G$ and its blocks.
\end{Remark}

\section{The $HH^1$ of blocks with nontrivial cyclic defect group}\label{PrE}

The theory of blocks with cyclic defect group is well understood (see for example \cite{CurRein, LinBlockII, SieWith}) and the $HH^1$ of such blocks has a Lie algebra structure that is particularly simple to describe.

Let $G$ be a finite group and $k$ an algebraically closed field of characteristic $p$ dividing $|G|$. If $B=kGb$ is a block of $kG$ with block idempotent $b$ and a defect group $P$, and if $C=kN_G(P)c$ is the Brauer correspondent of $B$ for some idempotent $c$, then we form the inertial quotient of $B$, the group $E=N_G(P,f)/PC_G(P)$ where $f$ is a block of $kC_G(P)$ satisfying $fc=f$. The inertial quotient is unique up to isomorphism, and so independent of choice of $f$. If $B=B_0$ is the principal block, then $\Br_P(b)$ is the principal block of of $kC_G(P)$, so that $f=c$ and $E=N_G(P)/PC_G(P)$. Note also that if $P$ is abelian, $PC_G(P)=C_G(P)$.

The inertial quotient appears frequently in the context of blocks with cyclic defect groups, in view of the results stated below. Throughout, unless otherwise stated $P$ will denote a nontrivial cyclic group of order $p$ and $E$ a $p'$-subgroup of $\Aut(P)$, of order $e$. Let $x$ and $y$ be generators of $P$ and $E$ respectively, then $E$ acts on $P$ and we will write the action as $y\cdot x=x^s$ for some $s=2,\ldots,p-1$, such that $s^e=1(\mod \ p)$. Let $H=P\rtimes E$ denote the semidirect product. 

The group algebra $kH$ is a Nakayama (also selfinjective serial) algebra (see \cite[\S 11.3]{LinBlockII}), and in particular, if $E$ is non-trivial, then by Linckelmann and Rubio y Degrassi (\cite[Example 5.7]{LinRubII}) $HH^1(kH)$ is a solvable Lie algebra with nilpotent derived Lie subalgebra.

\begin{Theorem}\cite[11.1.1/3/11]{LinBlockII}\label{stabeqThrm}
Let $G$ be a finite group and $B$ be a block of $kG$ with a nontrivial
cyclic defect group $P$ and inertial quotient $E$, and let $H=P\rtimes E$. Then the following hold. \begin{itemize}
\item[(i)] The inertial quotient $E$ is cyclic of order dividing $p-1$, and acts freely on $P$. 
\item[(ii)] There is an equality $|E|=|\IBr_k(B)|$.
\item[(iii)] There is a derived equivalence between $B$ and $kH$.
\end{itemize} 
\end{Theorem}

\noindent Derived equivalences preserve Hochschild cohomology so that in particular, with $B$ and $H$ as above, we have an isomorphism of restricted Lie algebras $HH^1(B)\cong HH^1(kH)$. The group $H$ is a Frobenius group by virtue of the free action of $E$ on $P$.

The next result is well-known and provides a characterisation of the $HH^1$ of blocks with a cyclic defect group in terms of the $HH^1$ of the group algebra $k(P\rtimes E)$ and the derivations of $HH^1(kP)$ that are fixed under the action of $E$ induced by the action on $P$. We point out that this theorem holds more generally for $P$ an elementary abelian $p$-group, and though the action of $E$ on $P$ is not necessarily Frobenius in this case the result may still be exploited as was seen in the proofs of Theorems \ref{ThrmM12}(i) and \ref{ThrmM22}(ii).

\begin{Theorem}\label{fixptisom}
Let $P$ be an elementary abelian $p$-group, $E\leq \Aut(P)$ and write $G=P\rtimes E$. Then the action of $E$ on $P$ induces an action on $HH^1(kP)$, and there is an inclusion of Lie algebras $$HH^1(kP)^E\hookrightarrow HH^1(kG).$$ Moreover if the action of $E$ on $P$ is a Frobenius action then this inclusion is an isomorphism.
\end{Theorem}

\begin{proof}
We show something slightly more general. Let $G$ be a finite group with $N$ a normal subgroup of $p'$-index, and recall that $HH^*(kG)\cong H^*(G;kG)$ where $G$ acts on $kG$ by conjugation. By \cite[III.10.4]{Brown}, $H^*(G;kG)\cong H^*(N;kG)^{G/N}$, which by definition equals $\Ext_{kN}^*(k,kG)^{G/N}$.

Since $N$ is of index coprime to $p$ we have a decomposition under the conjugation action, $kG=kN\oplus k(G\backslash N)$, such that each summand is stable under the action of $N$. One sees that \begin{align}
HH^*(kG) & \cong\Ext_{kN}^*(k,kN)^{G/N}\oplus\Ext_{kN}^*(k,k(G\backslash N))^{G/N} \nonumber\\
& =HH^*(kN)^{G/N}\oplus\Ext_{kN}^*(k,k(G\backslash N))^{G/N}. \nonumber\end{align}
Now if $N=P$ is an elementary abelian $p$-group and $G=P\rtimes E$, we have $HH^1(kP)^{G/P} \cong HH^1(kP)^E\subseteq HH^1(kG)$.

We also have that if the action of $E$ on $P$ is Frobenius then $k(G\backslash P)=\bigoplus_{y\in E\backslash\{1\}}kPy$ is a direct sum of projective $kP$-modules under the conjugation action, so that $\Ext_{kP}^1(k,k(G\backslash P))=0$, and the second statement follows.
\end{proof}

Let $d\in HH^1(kP)$ be a derivation and recall that $x,y$ generate $P,E$ respectively, with $y\cdot x=x^s$ for some $s=2,\ldots,p-1$. Then the action of $E$ on $HH^1(kP)$ in Theorem \ref{fixptisom} is given by $\ ^yd\in HH^1(kP)$, the derivation sending $x\mapsto y\cdot d(y^{-1}\cdot x)$.

\begin{Proposition}\label{elabdimprop}
Let $P$ be an abelian $p$-group, $r$ the rank of $P/\Phi(P)$, and $E$ be an abelian $p'$-subgroup of $\Aut(P)$ acting freely on $P\backslash\{1\}$. Set $H=P\rtimes E$. Then $$\dim_k(HH^1(kH))=r\cdot\left(\frac{|P|-1}{|E|}\right).$$
\end{Proposition}

\begin{proof}
Since $E$ acts freely and is abelian, we have the $\frac{|P|-1}{|E|}+|E|$ conjugacy classes of $H$ as described in \cite[\S 2]{SieWith}, generalised to the homocyclic case. Noting that $\dim_k(\Hom(P,k)) =r$, the result follows by the centraliser decomposition.
\end{proof}

\noindent From this we obtain the following simple corollary via the derived equivalence of Theorem \ref{stabeqThrm}.

\begin{Corollary}\label{dimprop}
Let $B$ be a block of $kG$ with a nontrivial cyclic defect group $P$ and nontrivial inertial quotient $E$. Then $\dim_k(HH^1(B))=\frac{|P|-1}{|E|}$.
\end{Corollary}

The following theorem provides us with a simple recipe to determine a Lie algebra basis of the $HH^1$ of blocks with cyclic defect group. 

\begin{Theorem}\label{preformula} Let $P$ be a cyclic group of order $q=p^t$ with a generator $x$, and $E$ a $p$'-subgroup of $\Aut(P)$ of order $e$, with a generator sending $x$ to $x^s$. For $m=0,\ldots,q-1$, let $D_m$ denote the derivation on $kP$ that sends $$x\longmapsto\sum_{i=0}^{e-1}s^ix^{(m-1)s^{-i}+1}.$$  Define an equivalence relation on $\{D_0,\ldots,D_{q-1}\}$ by $D_m\sim D_n$ if and only if $n=(m-1)s^{-i}+1$ for some $i=0,\ldots,e-1$. Then we have the following.
\begin{itemize}
    \item[(i)] If $E$ is non-trivial, then for all $m$, $D_m(x)\in J(kP)$.
    
    \item[(ii)] The set of equivalence class representatives of $\sim$ on $\{D_0,\ldots,D_{q-1}\}$ forms a $k$-basis for $HH^1(kP)^{E}$ as a Lie algebra.
    
    \item[(iii)] Let $p>2$, $t=1$ so that $P$ is cyclic of order $p$, and let $E$ have order $p-1$. Then $HH^1(kP)^{E}$ has a $p$-toral basis.
\end{itemize}

\end{Theorem}

\begin{proof}
The Automorphism group $\Aut(P)$ is isomorphic to the cyclic group $C_{\varphi(q)}$ where $\varphi(q)=q-p^{t-1}$. Whence $E$ is cyclic acting freely on $P$. Let $E$ have a generator $y$, so that $y\cdot x=x^s$. We have the Lie algebra isomorphism $HH^1(kH)\cong HH^1(kP)^E$ of Theorem \ref{fixptisom}, and we look for a basis of $HH^1(kP)^E$ of derivations $D\in HH^1(kP)$ such that $^{y^j}D(x^i)=D(x^i)$ for all $j=0,\ldots,e-1$ and $i=0,\ldots,q-1$. 

Let $\CB=\{d_0,\ldots,d_{q-1}\}$ be the basis of $HH^1(kP)$ given by the maps $d_m$ sending $x\mapsto x^m$.  Taking inspiration from \cite[Proposition 2.5.10]{LinBlockI} we look for a basis of $HH^1(kP)^E$ in the set of $E$-orbit sums of $\CB$. We denote by $\Tr_1^E:HH^1(kP)\to HH^1(kP)^E$ the relative trace map, and define $D_m=\Tr_1^E(d_m)$. This maps $$x\longmapsto\sum_{z\in E}z\cdot d_m(z^{-1}\cdot x)=\sum_{i=0}^{e-1}y^{-i}\cdot d_m(x^{s^i}).$$ One sees that each term in the sum may be expressed as $y^{-i}\cdot(s^ix^{m+s^i-1})=s^ix^{(m-1)s^{-i}+1}$, as required. We note that if $E$ is trivial then this formula recovers the basis $\CB$ of $HH^1(kP)$, for $D_m=d_m$ in this case.

Recall that $s$ is such that $s^e\equiv 1 \ (\mod \ p)$. One sees that $(1-s)(\sum_{i=0}^{e-1}s^i)=1-s^e=0$, and if $E$ is non-trivial, then the free action of $E$ on $P$ implies $s>1$, whence $\sum_{i=0}^{e-1}s^i=0$, from which it is clear that $D_1\equiv 0$ and $D_m(x)=\sum_{i=0}^{e-1}s^ix^{(m-1)s^{-i}+1}\in I(kP)=J(kP)$, proving (i). 

One verifies that $\sim$ is indeed an equivalence relation on $\{D_0,\ldots,D_{q-1}\}$, and that there is a bijection between the set of $E$-orbits on $P$ and the equivalence classes under $\sim$, via the map \begin{align}
  \{E\cdot x^j \ | \ j=0,\ldots,q-1\} &\leftrightarrow \{[D_m] \ | \ m=0,\ldots, q-1 \},   \nonumber \\
   \ E\cdot x^{m-1} &\mapsto [D_m], \nonumber
\end{align} proving (ii).

Now suppose that $p\neq 2$, $t=1$ so that $P$ is cyclic of order $p$, and that $E$ has order $p-1$. Fix an element $w=\sum_{i=1}^{p-1}i^{-1}x^i\in kP.$ Then $w\in I(kP)=J(kP)$, as $\sum_{i=1}^{p-1}i^{-1}=\sum_{i=1}^{p-1}i=0.$ On the other hand $w\notin J(kP)^2$. Indeed $J(kP)^2$ is generated as $k$-vector space by all elements of the form $x^i(x-1)^2$ as $i$ ranges from $0$ to $p-1$, and if we set $w=\sum_{i=0}^{p-1}\lmd_ix^i(x-1)^2$ then on comparing coefficients we get a system of linear equations in the variables $\lmd_i$ whose coefficient matrix has rank $p-2$ but whose augmented matrix has rank $p-1$, whence no such $\lmd_i\in k$ exist. By Nakayama's Lemma $w$ generates $J(kP)$ as a $kP$-module, and noting that $w^p=(\sum_{i=1}^{p-1}i^{-1}x^i)^p=\sum_{i=1}^{p-1}i^{-p}x^{pi}=\sum_{i=1}^{p-1}i^{-1}=0$ one sees that $\{1,w,w^2,\ldots,w^{p-1}\}$ is a basis for $kP$ as a $k$-vector space.

One checks that $w$ is sent to $sw$ under the action of $y$ where $s$ is such that $x$ is sent to $x^s$ by $y$ as before. By Proposition \ref{dimprop} $HH^1(kP)^E$ has dimension $1$ so we may choose $D_0$ as representative of the class forming a basis, without loss of generality. Then we have that $D_0(w)=w$. To see this, first note that $d_0(w)=\sum_{i=1}^{p-1}x^{i-1}$, and consider

$$D_0(w) = \sum_{z\in E} \ ^zd_0(w) = \sum_{i=0}^{p-2} y^i\cdot d_0(y^{-i}\cdot w)  = \sum_{i=0}^{p-2}y^i\cdot \left(s^{-i}d_0(w)\right) = \sum_{i=0}^{p-2}\sum_{j=1}^{p-1}s^{-i}x^{(j-1)s^i},$$ which on expanding the direct sum indexed by $j$ is seen to be equal to $\sum_{i=1}^{p-2}i^{-1}w$. One verifies that $\sum_{i=1}^{p-2}i^{-1}=1$ and the result then follows: $D^p_0$ sends $w$ to $w$.

\end{proof}

\begin{Remark}
In the previous result it is possible to show that for $P$ cyclic of order $q=p^t$, $t>1$, there is still a basis of $HH^1(kP)^E$ consisting of maps $D_m$ such that $D_m^p=D_m$. Indeed, fixing $w=\sum_{i=1,p\nmid i}^{q-1}i^{-1}x^i$ one verifies by an analogous proof that $\{1,w,w^2,\ldots,w^{q-1}\}$ is a basis for $kP$ as a $k$-vector space, and that $D_m^p(w)=D_m(w)=w$ for all $m=0,\ldots,q-1$, $m\neq 1$. On the other hand a simple verification shows that in general $[D_m,D_n]\neq0$ for $n\neq m$ and so this basis is not $p$-toral.
\end{Remark}

\section{The Tate-Hochschild cohomology dimension formula}\label{altsum}

Using a result of K\"ulshammer and Robinson \cite[Theorem 1]{KueRoHH} we verify the following statement from Theorem \ref{ThrmM12}(i). As mentioned before we use this alternative method for verification in order to demonstrate its use in application to other groups and more general cases, and its power as a computational tool.

\begin{Proposition}\label{M12prop}
    Let $G=M_{12}$, let $k$ be an algebraically closed field of characteristic 2 and let $B_0$ be the principal block of $kG$. Then $\dim_k(HH^1(B_0))=13$.
\end{Proposition}

From Table \ref{tabdefect} we see that $M_{12}$ is the only sporadic Mathieu group with a non-principal block that has defect greater than 1, occuring when $p=2$. The results of Section \ref{PrE} no longer apply here and so the method used to find the dimensions of the $HH^1$ of the principal block are different for this case only. Fortunately we have the result of Holm (\cite[2.1]{Holm}) to tell us that $\dim_k(HH^1(B_1))=2$ in this case, from which we can use the centraliser decomposition to calculate $\dim_k(HH^1(kG))=15$ and recover $\dim_k(HH^1(B_0))=13$.

On the other hand, we can verify this dimension more explicitly, using a formula for the dimension of the Tate-Hochschild cohomology given by K\"{u}lshammer and Robinson, developed from Kn\"orr and Robinson's refinement of Alperin's weight conjecture \cite{KnRo}. The results of this section demonstrate the weight of application of this formula, and further applications along the same vein have been used to verify the dimensions of the $HH^1$ of the principal blocks given in this paper (we omit these details). Throughout let $G$ be finite and $k$ be algebraically closed of characteristic $p$ dividing $|G|$.

In \cite{KnRo} Kn\"orr and Robinson consider several simplicial complexes related to the $p$-local structure of $G$, on which $G$ acts naturally by conjugation. Let $\CP_G$ denote the simplicial complex with $n$-simplices the chains of non-trivial $p$-subgroups of $G$ of the form $Q_1<Q_2<\cdots<Q_{n+1}$ (with strict inclusion). In the course of the proof of Proposition \ref{M12prop}, we will only consider $\CU_G\subset\CP_G$ of chains of the form $Q_1<Q_2<\cdots <Q_n$ such that $Q_i=O_p(N_G(Q_i))$ for all $i=1,\ldots,n$. We call such subgroups $Q_i$ \textit{radical p-subgroups}.

We denote by $\sigma$ an element of $\CP_G$ (also known as a $p$-chain), that is $\sigma = Q_0<\cdots<Q_m$ for some integer $m\geq0$, such that $Q_i$ is a $p$-subgroup of $G$ for all $0\leq i\leq m$. The empty chain is viewed as a $-1$-simplex, and so for this and also for notational convenience all chains will be viewed as chains of the form $Q_0<Q_1<\cdots<Q_m$ where it is always the case that $Q_0=\{1_G\}$ (c.f \cite[Section 2]{KnRo}).

Given a $p$-chain $\sigma$, we define the subgroups $C_G(\sigma)=C_G(Q_m)$ and $N_G(\sigma)=\cap_{i=0}^mN_G(Q_i)$. The \textit{length} of $\sigma$ denoted $|\sigma|$, is equal to the number of nontrivial subgroups in the chain: in our case it will always be equal to $m$ as our chains always start with the trivial subgroup. We say that $\sigma$ is a \textit{radical} or \textit{normal} chain if $Q_0=O_p(G)$ and $Q_i=O_p(N_G(\sigma_i))$, where $\sigma_i=Q_0<\cdots<Q_i$ for all $0\leq i\leq m$. Denote by $\CR_G$ the set of all radical $p$-chains of $G$ (note this is not a simplicial complex, see the remarks following \cite[3.4]{KnRo}). 

Since $G$ acts by conjugation on $\CP_G$ and all its subsets we may consider $\CP_G/G$, a set of $G$-conjugacy class representatives of $p$-chains, with the analogous notation defined on the subsets of $\CP_G$ considered above. 

For a $p$-subgroup $Q$ of $G$, denote by $\Br_Q:(kG)^Q\to kC_G(Q)$ the Brauer homomorphism, sending $\sum_{g\in G}\lmd_gg$ to $\sum_{g\in C_G(Q)}\lmd_gg$. With this notation, we have the following result (restricted to positive degrees for Hochschild cohomology, c.f \cite[10.7.11]{LinBlockII}), in which we may replace $\CP_G/G$ with either $\CU_G/G$ or $\CR_G/G$, without altering the result.

\begin{Theorem}\cite[Theorem 1]{KueRoHH}\label{dimThrm}
Let $G$ be a finite group and $b$ a block idempotent of $kG$. For any chain $\sigma=Q_0 < Q_1 < \cdots < Q_m$ in $\CP_G$ set $B_\sigma = kN_G(\sigma)\Br_{Q_m}(b)$. Then for any integer $n>0$ we have $$\sum_{\sigma\in\CP_G/G}(-1)^{|\sigma|}\dim_k(HH^n(B_\sigma))=0.$$
\end{Theorem}

In this alternating sum, the chain $\sigma=Q_0$ recovers the dimension $\dim_k(HH^1(kGb))$ which we may isolate on one side of the equality, and the other terms require only calculations of the dimensions of the $HH^1$ of smaller group algebras, or blocks of smaller group algebras, for which we can use the centraliser decomposition or Lemma \ref{c2xs5} below.

Since Proposition \ref{M12prop} concerns the principal block of the $M_{12}$, we will make use of the following when evaluating the Brauer homomorphism involved in the structure of the $B_\sigma$ above.

\begin{Theorem}\cite[6.3.14, Brauer's Third Main Theorem]{LinBlockII}\label{BrauerThrm3}
Let $G$ be a finite group and $b$ the principal block idempotent of $kG$. Then, for any $p$-subgroup $Q$ of $G$, $\Br_Q(b)$ is the principal block of both $kC_G(Q)$ and $kN_G(Q)$.
\end{Theorem}

\noindent Similarly, the following result will often be used to show that the  $\Br_{Q_m}(b)$ in the structure of the $B_\sigma$ above is simply equal to the identity of $kG$

\begin{Lemma}\label{pgroupLemma}\cite[1.11.4]{LinBlockI}
Let $G$ be a $p$-group and $\chr(k)=p$. Then $kG$ is indecomposable.
\end{Lemma}

For the remainder of this section we fix $G=M_{12}$, $\chr(k)=2$, and $B_0=kGb$ the principal 2-block of $kG$, with block idempotent $b$. In order to prove Proposition \ref{M12prop} we will calculate some auxiliary $HH^1$ dimensions that appear in the alternating sum. In particular, $G$ has a unique conjugacy class of subgroups isomorphic to $C_2\times S_5$, to $S_3\times A_4$ and to $C_2^2\times S_3$, and the principal block of these groups arises in the alternating sum as some $kN_G(\sigma)\Br_Q(b)$ for some chain $\sigma$ and $2$-subgroup $Q$ of $G$.

\begin{Lemma}\label{c2xs5}
Let $H$ be one of $C_2\times S_5, \ S_3\times A_4$ or $C_2^2\times S_3$. Then $kH$ has two blocks $B_0'$ and $B_1'$ and we have the following dimensions.
\begin{itemize}
    \item[(i)] Let $H=C_2\times S_5$ Then $\dim_k(HH^1(B_0'))=22$ and $\dim_k(HH^1(B_1'))=8.$
    \item[(ii)] Let $H=S_3\times A_4$. Then $\dim_k(HH^1(B_0'))=12$ and $\dim_k(HH^1(B_1'))=2.$
    \item[(iii)] Let $H=C_2^2\times S_3$. Then $\dim_k(HH^1(B_0'))=24$ and $\dim_k(HH^1(B_1'))=8.$ 
\end{itemize}
\end{Lemma}

\begin{proof}

First let $H=C_2\times S_5$. Blocks of $H$ are in bijection with blocks of $kC_2\ten_k kS_5$, and $kS_5$ decomposes into two blocks $B_0(kS_5)$ and $B_1(kS_5)$ with defect groups isomorphic to $D_8$ and $C_2$ respectively. The block $B_0(kS_5)$ is then a tame block of type $D(2\CA)$, derived equivalent to $kS_4$ by Holm \cite{HolmDer}, which by the centraliser decomposition has $\dim_k(HH^1(kS_4))=6$ and has $\dim_k(HH^0(kS_4))=\dim_k(Z(kS_4))=5.$ One sees therefore that $HH^1(B_0')\cong HH^1(k(C_2\times S_4))$ and the latter has an additive decomposition isomorphic to $kC_2\ten_kZ(kS_4)\oplus kC_2\ten_k HH^1(kS_4)$ by the K\"unneth formula (see \cite[2.21.7]{LinBlockI}), whence the dimension of $HH^1(B_0')$ is 22. 

Similarly, the block $B_1(kS_5)$ is derived equivalent to $kC_2$ by Theorem \ref{stabeqThrm}(iii) so that $HH^1(kB_0')\cong HH^1(k(C_2\times C_2))$ which has dimension 8.  Using the centraliser decomposition one verifies that $HH^1(kH)$ has dimension 30. 

Now let $H=S_3\times A_4$. The proof follows in the same vein as (i). The group algebra $kA_4$ is indecomposable whereas $kS_3$ has 2 blocks $B_0(kS_3)$ and $B_1(kS_3)$ with defect 1 and 0 respectively, isomorphic to $kC_2$ and $M_n(k)$ respectively, for some $n$. Whence $B_0'\cong B_0(kS_3\tenk kA_4)\cong k(C_2\times A_4)$ which has $HH^1$ of dimension 12 by the centraliser decomposition. Similarly $B_1'\cong kA_4$ which has $HH^1$ of dimension 2, and using the centraliser decomposition one verifies that $HH^1(kH)$ has dimension 14.

Finally let $H=C_2^2\times S_3$. One sees that $B_0'\cong B_0(kC_2^2\tenk kS_3)\cong kC_2^2\tenk B_0(kS_3)\cong kC_2^3$ which has $HH^1$ of dimension 24, $B_1'\cong kC_2^2$ which has $HH^1$ of dimension 8, and using the centraliser decomposition one verifies that $HH^1(kH)$ has dimension 32.
\end{proof}

\begin{proof}[Proof of Proposition \ref{M12prop}]
We will apply K\"ulshammer and Robinson's result using $\CU_G$, the complex of radical 2-subgroups of $G$. A defect group of $B_0$ is given by $P\cong C_4^2\rtimes C_2^2$ (see \texttt{SmallGroup(64,134)} in the \texttt{GAP} library), and using \cite[Table 2D]{AnCon} we deduce the information displayed in Table \ref{tab2}, which gives us the radical 2-subgroups of $G$, up to $G$-conjugacy (the details of the actions of the semidirect product may be found using \texttt{GAP} and the database \cite{DokGrp}).

\begin{table}[h]
\centering
\caption{The radical 2-subgroups of $M_{12}$}\label{tab2}
\begin{tabular}{l|p{0.14\linewidth}|p{0.14\linewidth}|p{0.14\linewidth}}
Label, $Q$  &  Isomorphic to           &  \hfil $N_G(Q)$   		    &  \hfil $C_G(Q)$  		    \\ 
\hline
\hfil $P_0$ &  \hfil $\{1_G\}$      		   &  \hfil $G$         	     &  \hfil $G$        	      \\
\hfil $P_1$ &  \hfil $C_2$          		   &  \hfil $C_2\times S_5$      &  \hfil $C_2\times S_5$     \\
\hfil $P_2$ &  \hfil $C_2\times C_2$ 		   &  \hfil $S_3\times A_4$      &  \hfil $C_2^2\times S_3$   \\
\hfil $P_3$ &  \hfil $C_4\rtimes D_8$          &  \hfil $C_4^2\rtimes D_{12}$ &  \hfil $C_2\times C_2$     \\
\hfil $P_4$ &  \hfil $2_+^{1+4}$               &  \hfil $Q_8\rtimes S_4$     &  \hfil $C_2$     		  \\
\hfil $P$   &  \hfil $ C_4^2\rtimes C_2^2$          &  \hfil $P$    		         &  \hfil $C_2$   	           
\end{tabular}
\end{table}

The $G$-conjugacy classes of subgroups of $P$ and their lattice are given in Figure \ref{M12lattice}, where reading left to right a connecting line denotes strict inclusion.
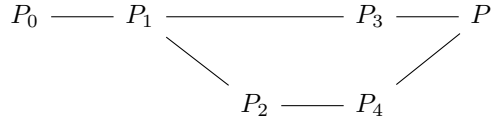
\begin{figure}[h]
\caption{Lattice of radical 2-subgroups of $M_{12}$}\label{M12lattice}
\[\begin{tikzcd}
	{P_0} & {P_1} & & {P_3} & P \\
	&& {P_2} & {P_4}
	\arrow[no head, from=1-1, to=1-2]
	\arrow[no head, from=1-2, to=2-3]
	\arrow[no head, from=2-3, to=2-4]
	\arrow[no head, from=2-4, to=1-5]
	\arrow[no head, from=1-2, to=1-4]
	\arrow[no head, from=1-4, to=1-5]
\end{tikzcd}\]
\end{figure}
We therefore have that $\CU_G/G$ consists of 20 chains of radical 2-subgroups (all starting at $P_0=\{1_G\}$); one of length 0, five of length 1, eight of length 2, 5 of length 3 and one of length 4. From these 20 chains, we obtain the  information displayed in Table \ref{tab3}. We note that the $B_{\chi}$ are found using Theorem \ref{BrauerThrm3}, Lemma \ref{pgroupLemma} and \texttt{GAP}, the dimensions $\dim_k(HH^1(B_\chi))$ for $\chi\in\{\beta,\gamma,\eta\}$ are calculated in Lemma \ref{c2xs5}, and all other dimensions are calculated using the centraliser decomposition.

\begin{table}[h]
\centering
\caption{Representatives of the classes of $\CU_G/G$}\label{tab3}
\begin{tabular}{l|p{0.28\linewidth}|p{0.28\linewidth}|p{0.2\linewidth}}
 $\chi$  & \hfil Representative          &  \hfil $B_\chi$   		    &  \hfil $\dim_k(HH^1(B_\chi))$  		    \\ 
\hline
\hfil $\alpha$   &  \hfil $P_0$                &  \hfil $B_0(kG)=kGb$         	          &  \hfil      ?  	      \\
\hfil $\beta$    &  \hfil $P_0<P_1$            &  \hfil $B_0(k(C_2\times S_5))$         &  \hfil  	22      \\
\hfil $\gamma$   &  \hfil $P_0<P_2$            &  \hfil $B_0(k(S_3\times A_4))$  &  \hfil 		12 	      \\
\hfil $\delta$   &  \hfil $P_0<P_3$            &  \hfil $k(C_4^2\rtimes D_{12})$		  &  \hfil 		30	      \\
\hfil $\epsilon$ &  \hfil $P_0<P_4$            &  \hfil $k(Q_8\rtimes S_4)$    		  &  \hfil      23		  \\
\hfil $\zeta$    &  \hfil $P_0<P$              &  \hfil $kP$    		       			  &  \hfil    	40          \\
\hfil $\eta$     &  \hfil $P_0<P_1<P_2$        &  \hfil $B_0(k(C_2^2\times S_3))$       &  \hfil   	24          \\
\hfil $\theta$   &  \hfil $P_0<P_1<P_4$        &  \hfil $k(C_2\times D_8)$    	      &  \hfil    	28          \\
\hfil $\iota$    &  \hfil $P_0<P_1<P$          &  \hfil $k(C_2\times D_8)$    	      &  \hfil    	28          \\
\hfil $\kappa$   &  \hfil $P_0<P_1<P_3$        &  \hfil $k(C_2\times S_4)$    	      &  \hfil    	22          \\
\hfil $\lambda$  &  \hfil $P_0<P_2<P_4$        &  \hfil $k(C_2\times A_4)$    	      &  \hfil    	12          \\
\hfil $\mu$      &  \hfil $P_0<P_2<P$          &  \hfil $kC_2^3$    		       		  &  \hfil    	24          \\
\hfil $\nu$      &  \hfil $P_0<P_4<P$          &  \hfil $kP$    		        		  &  \hfil    	40          \\
\hfil $\xi$      &  \hfil $P_0<P_3<P$          &  \hfil $kP$    		    		      &  \hfil    	40          \\
\hfil $\pi$      &  \hfil $P_0<P_1<P_2<P_4$    &  \hfil $kC_2^3$    		     	      &  \hfil    	24          \\
\hfil $\rho$     &  \hfil $P_0<P_1<P_2<P$      &  \hfil $kC_2^3$    		     	      &  \hfil    	24          \\
\hfil $\sigma$   &  \hfil $P_0<P_1<P_4<P$      &  \hfil $k(C_2\times D_8)$              &  \hfil    	28          \\
\hfil $\tau$     &  \hfil $P_0<P_1<P_3<P$      &  \hfil $k(C_2\times D_8)$              &  \hfil    	28          \\
\hfil $\phi$     &  \hfil $P_0<P_2<P_4<P$      &  \hfil $kC_2^3$    		     		  &  \hfil    	24         \\
\hfil $\psi$     &  \hfil $P_0<P_1<P_2<P_4<P$  &  \hfil $kC_2^3$    		         	  &  \hfil    	24          
\end{tabular}
\end{table}

Let $\Lambda=\{\beta,\gamma,\ldots,\psi\}$, that is, the set of all chain labels $\chi$ in Table \ref{tab3}, excluding $\chi=\alpha$. Then the alternating sum formula of Theorem \ref{dimThrm} tells us that $$ \dim_k(HH^1(B_\alpha))=\dim_k(HH^1(B_0(kG)))=\sum_{\chi\in\Lambda}(-1)^{|\chi|+1}\dim_k(HH^1(B_\chi))=13.$$ This completes the proof.

\end{proof}

\appendix 

\section{The dimension of the $HH^1$ of a finite group algebra}\label{Appcode}

The much celebrated centraliser decomposition allows us to determine the dimensions of the $HH^1$ of a finite group algebra with relative ease. It's heavy use throughout the calculations and proofs in this paper confirms its utility, and so we explain briefly where these dimensions come from, as well as detail the \texttt{GAP} code used to efficiently implement the centraliser decomposition and compute these dimensions.

\begin{Proposition}\cite[Theorem 2.11.2(ii)]{BensonII}\label{CentraliserDecompProp}
	As a $k$-vector space, the additive structure of the Hochschild cohomology of $kG$ is given by
	\[HH^*(kG)\cong\bigoplus_{x\in G/\sim}H^*(C_G(x);k)\]
	where $G/\sim$ is the set of conjugacy class representatives of $G$ and $C_G(x)$ is the centraliser of $x\in G$. 
\end{Proposition}

\noindent The first cohomology group of $G$ with coefficients in $k^\times$, $H^1(G;k^\times)$, will always be considered with trivial action of $G$ on $k^\times$. One therefore sees that
$$
HH^1(kG)\cong\bigoplus_{x\in G/\sim}\Hom(C_G(x),k)
$$
where each summand is a finite abelian group of order dividing $|C_G(x)|$, and in particular when viewed as $k$-modules, $\dim_k(HH^1(kG))=\sum_{x\in G/\sim}\dim_k(\Hom(C_G(x),k))$. 

There exists a twisted group algebra analogue to the centraliser decomposition

\begin{Proposition}\cite[Lemma 3.5]{WithSpoon}\label{twistedcent}
There is an additive decomposition $$HH^1(k_\alpha G)\cong \bigoplus_{x\in G/\sim} H^1(C_G(x);V_{x,\alpha}),$$ where $V_{x,\alpha}$ is a $1$-dimensional $C_G(x)$-module spanned by the image $\hx\in k_\alpha G$ of $x\in G$, with action given by $g\cdot \hx=\alpha(g,x)\alpha(x,g)^{-1}\hx$ for all $g\in C_G(x)$.
\end{Proposition}

To calculate the individual dimensions of $\Hom(C_G(x),k)$ in Proposition \ref{CentraliserDecompProp} we use the next result, which then forms the rest of the \texttt{GAP} code.

\begin{Lemma}\label{quotLemma}
	The $k$-vector space $\Hom(G,k)$ is nonzero if and only if $G$ has a quotient isomorphic to a nontrivial $p$-group. In particular, $\dim_k(\Hom(G,k))=\dim_k(\Hom(R/\Phi(R),k))$, where $R=G/O^p(G)$.
\end{Lemma}

We detail below the simple \texttt{GAP} commands to determine the dimensions of the $HH^1$ of a finite group algebra (in particular a group on which \texttt{GAP} is able to perform calculations), using the centraliser decomposition and Lemma \ref{quotLemma}. 

Once a group $G$ is defined in \texttt{GAP}, the command \texttt{div} below gives the set of prime divisors of $|G|$, $\pi(G)$. The functions following \texttt{div} then take as input either $G$, or $G$ and a prime $p\in\pi(G)$, and output the following lists;\begin{itemize}
    \item[1.] The centralisers of a complete set of conjugacy class representatives of $G$, 
    \item[2.] For each conjugacy class representative $x\in G/\sim$, the groups $C_G(x)/O^p(C_G(x))=R$,
    \item[3.] The elementary abelian $p$-groups $R/\Phi(R)$,
    \item[4.] The rank of $R/\Phi(R)$,
\end{itemize} respectively. The final function then produces a list of dimensions $\dim_k(HH^1(kG))$, with one entry for each $p\in \pi(G)$. 

\vspace{2mm}
\verb|div:=PrimeDivisors(Size(G));|

\vspace{2mm}
\verb|ListFpCentralisers:=function(X);|

\verb|> return List(ConjugacyClasses(X),i->|

\hspace{16mm}\verb|Image(IsomorphismFpGroup(Centralizer(X,Representative(i)))));|

\verb|> end;;|

\vspace{2mm}
\verb|ListMaxPQuot:=function(X,p);|

\verb|> return List(ListFpCentralisers(X),i -> Image(EpimorphismPGroup(i,p)));|

\verb|> end;;|

\vspace{2mm}
\verb|ListElAb:=function(X,p);|

\verb|> return List(ListMaxPQuot(X,p), i -> i/FrattiniSubgroup(i));|

\verb|> end;;|

\vspace{2mm}
\verb|ListPRank:=function(X,p);|

\verb|> return List(ListElAb(X,p), i -> RankPGroup(i));|

\verb|> end;;|

\vspace{2mm}
\verb|dimHH1:=function(X);|

\verb|> return List(div,i->Sum(ListPRank(X,i)));|

\verb|> end;|

\vspace{5mm}
\noindent\textbf{Acknowledgements.} I would like to thank Xin Huang for his helpful comments on an earlier version of this paper, and my PhD supervisor Markus Linckelmann for his support, and for many interesting conversations.

\end{document}